\documentclass[12pt,reqno, a4paper]{amsart}

\usepackage{amsmath,amsthm,amsfonts,amssymb,amsxtra,appendix,bookmark,dsfont,bm,color}
\usepackage[shortlabels]{enumitem}
\usepackage{graphicx}
\usepackage[utf8]{inputenc}

\usepackage{tcolorbox}
\theoremstyle{plain}
\newtheorem{theorem}{Theorem}
\newtheorem{lemma}[theorem]{Lemma}

\theoremstyle{remark}
\newtheorem{remark}{Remark}
\theoremstyle{definition}

\newcounter{step} 

\usepackage{etoolbox}
\AtBeginEnvironment{proof}{\setcounter{step}{0}}

\def\bq{\begin{eqnarray}}
\def\eq{\end{eqnarray}}
\def\bqq{\begin{align*}}
\def\eqq{\end{align*}}

\def\nn{\nonumber}

\def\eps{\varepsilon}
\def\wto{\rightharpoonup}

\renewcommand{\Re}{\operatorname{Re}}

\newcommand\1{{\ensuremath {\mathds 1} }}

\def\R {\mathbb{R}}

\def\R {\mathbb{R}}

\def\d{{\, \rm d}}

\title{Minimizing sequences of Sobolev inequalities revisited}

\author[C. Dietze]{Charlotte Dietze}
\address[C. Dietze]{Department of Mathematics, LMU Munich, Theresienstrasse 39, 80333 Munich, Germany;\newline
 Institut des Hautes Etudes Scientifiques, 35 Route de Chartres, Bures-sur-Yvette, France} 
\email{dietze@math.lmu.de}

\author[P.T. Nam]{Phan Th\`anh Nam}
\address[P.T. Nam]{Department of Mathematics, LMU Munich, Theresienstrasse 39, 80333 Munich, Germany; \newline Munich Center for Quantum
Science and Technology (MCQST),  Schellingstr. 4, 80799 Munich, Germany} 
\email{nam@math.lmu.de}

\begin{document}

\begin{abstract} We give a new proof of the compactness of minimizing sequences of the Sobolev inequalities in the critical case. Our approach relies on a simplified version of the concentration-compactness principle, which does not require any refinement of the Sobolev embedding theorem. 
\end{abstract}

\maketitle


\section{Introduction}

The classical Sobolev inequality \cite{Sobolev-38} asserts that for all $d\ge 2$ and $1<p<d$, 
\begin{align}\label{eq:Sobolev}
\int_{\R^d} |\nabla u(x) | ^p \d x \ge S_{d,p} \left( \int_{\R^d} |u(x)|^{p^*} \d x \right)^{p/p^*}, \quad \forall u \in \dot{W}^{1,p}(\R^d)
\end{align}
where $p^*=dp/(d-p)$ and the constant $S_{d,p} >0$ is independent of $u$. The sharp constant $S_{p,d}$ in \eqref{eq:Sobolev} was determined by Rosen \cite{Rosen-71} for $p=2$ and $d=3$, and this was subsequently extended to the general case by Aubin \cite{Aubin-76} and Talenti \cite{Talenti-76}. For $p=2$, the determination of Sobolev optimizers is also a consequence of Lieb's work on the Hardy--Littlewood--Sobolev inequality \cite{Lieb-83}. These results crucially rely on rearrangement inequalities, which allow to restrict the problem to radially symmetric decreasing functions. There are also elegant free-rearrangement methods which allow to determine directly all optimizers \cite{CarCarLos-10,FraLie-10,FraLie-12,FraLie-12b}.

In this paper, we are interested in not only the existence of optimizers  but also the compactness of minimizing sequences modulo relevant symmetries of the variational problem
\begin{align*}
S_{d,p} = \inf \left\{  \frac{ \| \nabla u\|_{L^p(\R^d)}^p } {\|u\|_{L^{p^*}(\R^d)}^p}  \,|\, u \in \dot{W}^{1,p}(\R^d),\, u\not\equiv 0 \right\}, 
\end{align*}
which is important in many problems of nonlinear and stability analysis. In this context, Lions \cite{Lions-85} established the following foundational result using the concentration-compactness method.

\begin{theorem} \label{thm:1} For $d\in \mathbb{N}$ and $1<p<d$, the Sobolev inequality \eqref{eq:Sobolev} has an optimizer $u \in \dot{W}^{1,p}(\R^d)$. Moreover, any minimizing sequence in $u \in \dot{W}^{1,p}(\R^d)$ is, up to translations and dilations, relatively compact in $\dot{W}^{1,p}(\R^d)$. 
\end{theorem}

Our aim is to give a short proof of Theorem \ref{thm:1} by introducing a  simplified version of Lions' concentration-compactness principle. Lions’s original argument in \cite{Lions-85} addresses not only minimizing sequences but also general sequences which converge weakly in $\dot{W}^{1,p}(\R^d)$. The minimizing property is used only at the final stage, at the level of limiting measures. The concentration-compactness method has been further developed to investigate profile decompositions of weakly convergent sequences, beginning with the works of Gérard \cite{Gerard-98} and Jaffard \cite{Jaffard-99}. In contrast, our approach employs the minimizing property at an earlier stage, allowing us to establish tightness and exclude delta-type concentration from the beginning, thereby significantly simplifying the analysis.

\medskip
In Section \ref{sec:p=2}, we will consider the case $p=2$ where the Hilbert space structure simplifies many aspects of the problem. In this case, the main task is to show that any minimizing sequence ${u_n}$, up to translations and dilations, has a {\em non-trivial} weak limit $u_0$ in $\dot{H}^1(\R^d)$. This task is typically achieved by refining the Sobolev inequality; see e.g. \cite{GerMeyOru-97,PalPis-14} 
and the lecture notes \cite[Section 6]{Lewin-10}, \cite[Section 3]{Frank-11}, \cite[Section 4.2]{KilVis-13} (see also \cite[Appendix B]{LenLew-11} for the fractional case). Our proof, however, requires no additional input beyond the Sobolev compact embedding $L^2_{\rm loc} (\R^d) \subset \dot{H}^1(\R^d)$, which is sufficient to enable a localization method and rules out the possibility of a zero limit. For the latter, we use a convexity argument based on the elementary inequality 
\begin{align}\label{eq:strict-binding}
a^{p/p^*}+b^{p/p^*}  \ge  (1-\eps)^{-p/p^*} (a+b)^{p/p^*},\quad \forall a,b\in [\delta,1],
\end{align}
where the constant $\eps\in (0,1)$ depends only on $p/p^* \in (0,1)$ and $\delta\in (0,1/2)$.  Once we know $u_0 \neq 0$, the conclusion follows from standard arguments. In particular, we can use the weak convergence $\nabla u_n\wto \nabla u_0$ in $L^2(\R^d)$ to split the kinetic energy \begin{align}
\label{eq:S-p2-1}
\int_{\R^d} \left(  |\nabla u_n|^2 - |\nabla  u_0|^2 - |\nabla (u_n-u_0)|^2  \right) = 2\Re \langle \nabla u_0, \nabla(u_n-u_0)\rangle_{L^2} \to 0
\end{align}
which implies that $u_n-u_0$ is also a minimizing sequence by the same convexity argument. Thus we conclude $u_n-u_0\to 0$ strongly by the previous argument. 

\medskip
In Section \ref{sec:general-p}, we give a proof of Theorem \ref{thm:1} for the general case $1<p<d$. This is significantly more challenging since even if we know that the minimizing sequence $u_n$ has a weak limit  $u_0\ne 0$ in $\dot W^{1,p}$, it is not obvious why $u_0$ is a minimizer. The reason is that the extension of \eqref{eq:S-p2-1} does not hold under the weak convergence $\nabla u_n\wto \nabla u_0$ in $L^p(\R^d)$, see e.g. \cite[Eq. (4.23)]{KilVis-13}. It requires a deeper revision of the concentration-compactness argument. In particular, we have the following result:

\begin{lemma}\label{lem:concentration} Let $1<p<d$. Assume that $v_n \wto 0$ weakly in $\dot{W}^{1,p}(\R^d)$. Then at least one of the following holds true: $|v_n|^{p^*}$ vanishes locally, namely
\begin{align}\label{eq:vn-Sobolev-phi-conclusion-1}
\limsup_{n\to \infty}\int_{B_R} |v_n|^{p^*} =0, \quad \forall R>0,  
\end{align}
or there exists a delta-type concentration, namely 
\begin{align}\label{eq:p-vn-delta} 
\lim_{r\to 0} \limsup_{n\to \infty} \sup_{x\in \R^d}\int_{B_r(x)} |v_n|^{p^*} >0. 
\end{align} 
\end{lemma}

This result is a simplified version of \cite[Lemma I.1]{Lions-85}, for which we can provide a simple proof. To arrive at the conclusion, an additional argument is required  to establish the tightness and the absence of delta-type concentration for minimizing sequences. Both of these can be derived using the convexity argument via \eqref{eq:strict-binding} again.

\medskip

Our method can be extended to handle the fractional Sobolev inequalities 
\begin{align}\label{eq:Sobolev-fractional}
\| u\|_{\dot W^{s,p}(\R^d)}^p \ge S_{d,s,p} \left( \int_{\R^d} |u(x)|^{p^*} \d x \right)^{p/p^*}, \quad \forall u \in \dot{W}^{s,p}(\R^d)
\end{align}
for all $d\ge 1$, $d>s>0$ and $d/s>p>1$. Here $p^*=dp/(d-ps)$ and 
$$
\| u\|_{\dot W^{s,p}(\R^d)}^p = \begin{cases} \int_{\R^d} |D^m u(x)|^p \d x \quad \text { if } s=m\in \mathbb{N}, \\
 \int_{\R^d}\int_{\R^d} \dfrac{|D^m u(x)-D^m u(y)|^p}{|x-y|^{d+p\sigma}} \d x \d y \quad \text{ if } s=m+\sigma, m \in \mathbb{N}, \sigma\in (0,1).    
\end{cases}
$$
We denote $D^m u= (D^\alpha u)_{|\alpha|=m}$ for multi-indexes $\alpha\in \{0,1,...\}^d$ and $|D^m u|$ can be any norm in this product vector space (e.g., the Euclidean norm $(\sum_{|\alpha|=m} |D^\alpha u|^2)^{1/2}$).

\medskip

In Section  \ref{sec:general-p-s} we will prove the following result. 
 
\begin{theorem} \label{thm:2} For $d\in \mathbb{N}$, $d>s>0$ and $d/s>p>1$, the Sobolev inequality \eqref{eq:Sobolev-fractional} has an optimizer $u \in \dot{W}^{s,p}(\R^d)$. Moreover, any minimizing sequence in $\dot{W}^{s,p}(\R^d)$ is, up to translations and dilations, relatively compact in $\dot{W}^{s,p}(\R^d)$. 
\end{theorem}

This result is well-known. The case $1< s \in \mathbb{N}$ was already covered by Lions \cite{Lions-85}, which is particularly interesting since rearrangement techniques do not apply. The fractional case $s\in (0,1)$ was studied  by Zhang \cite{Zhang-21}, based on a refinement of Sobolev inequality in Morrey and Besov spaces, in the same spirit as by Palatucci and Pisante in \cite{PalPis-14}. We also refer to the work of Bahouri, Cohen and Koch \cite{BahCohKoc-11} for a general profile decomposition covering a larger range  of critical embedding of function spaces, including $\dot{W}^{s,p}(\R^d)$. 

Our proof of Theorem \ref{thm:2} in the integer case $s\in \mathbb{N}$ proceeds in the same way as the proof of Theorem \ref{thm:1} in Section \ref{sec:general-p}, where we only need a straightforward extension of Lemma \ref{lem:concentration}. The proof of Theorem \ref{thm:2} in the fractional case $s\notin \mathbb{N}$ is conceptually easier since the approach for $p=2$ applies. The reason is that an analogue of \eqref{eq:S-p2-1} for $\dot W^{s,p}(\R^d)$ holds due to the Brezis--Lieb lemma  \cite{BreLie-83}, which was also used in \cite{Zhang-21} for $0<s<1$.  

Finally, we refer to the recent work of Frank and Loss  \cite{FraLos-22} for an alternative proof of the compactness of Sobolev minimizing sequences using  Ekeland's variational principle. Moreover, we note that the compactness of minimizing sequences also holds for the sub-critical Sobolev embedding $W^{s,p}(\R^d) \subset L^q(\R^d)$, with $p < q < p^*$, where Lions' argument \cite{Lions-84b} is  simpler to apply.

\bigskip
\noindent
{\bf Notation:} We will denote by $B_R(x)$ the ball $\{y: |y-x|<R\}$ in $\R^d$, and write $B_R$ for $B_R(0)$. We write  $A\lesssim_d B$ if $A\le C_d B$ for a constant $C_d$ depending only on $d$. 

\bigskip
\noindent
{\bf Acknowledgements:} We would like to thank Rupert L. Frank, Patrick Gérard, Mathieu Lewin and Hoai-Minh Nguyen for helpful discussions and remarks. We thank the referee for valuable suggestions. Part of the work has been done when CD was a visiting reseacher at the Institut des Hautes Etudes and she would like to thank Laure Saint--Raymond for her support and hospitality. Partial support from 
the European Research Council via the ERC CoG RAMBAS (Project-Nr. 10104424) and from the Deutsche Forschungsgemeinschaft (DFG, German Research Foundation) via  the Germany’s
Excellence Strategy  (EXC-2111-390814868) and via TRR 352 (Project-ID 470903074) is acknowledged. CD also acknowledges the support from
Fondation Engie.

\section{The simplest case} \label{sec:p=2}

In this section we focus on the special case $p=2$ where \eqref{eq:Sobolev} states that 
\begin{align}\label{eq:Sobolev-p2}
\int_{\R^d} |\nabla u(x) |^{2} \d x \ge S_{d} \left( \int_{\R^d} |u(x)|^{2^*} \d x \right)^{2/2^*}, \quad \forall u \in \dot{H}^{1}(\R^d)
\end{align}
with $d\ge 3$ and $2^*=2d/(d-2)$. In this case, there exist refinements of \eqref{eq:Sobolev-p2} which help to prove the existence of optimizers \cite{GerMeyOru-97,PalPis-14,Lewin-10,LenLew-11,Frank-11,KilVis-13}. 
Our new proof below uses only elementary arguments. 

\begin{proof}[Proof of Theorem \ref{thm:1} for $p=2$] Fix $\delta_d>0$ small such that the ball $B_2 \subset \R^d$ is covered by less than $1/(2\delta_d)$ unit balls.   For any minimizing sequence $\{u_n\}_{n=1}^\infty \subset \dot{H}^{1}(\R^d)$, up to a translation,  dilation and multiplication by a constant, we can assume that
\begin{align}\label{eq:proof-simple=p2-input-1}
\int_{\R^d} |\nabla u_n|^2 \to S_d, \quad \int_{\R^d} |u_n|^{2^*}=1, \quad \int_{B_1} |u_n|^{2^*} = \sup_{x\in \R^d} \int_{B_1(x)} |u_n|^{2^*} = \delta_d
\end{align}
for all $n\in \mathbb{N}$. Then by the choice of $\delta_d$, 
\begin{align}\label{eq:proof-simple=p2-input-2}
\int_{B_2^c} |u_n|^{2^*} = 1- \int_{B_2} |u_n|^{2^*}  \ge 1- \frac{1}{2\delta_d} \sup_{x\in \R^d} \int_{B_1(x)} |u_n|^{2^*} \ge \frac{1}{2}. 
\end{align}
By the Banach--Alaoglu and the Sobolev embedding theorems, up to a subsequence, we can assume that there exists $u_0\in \dot{H}^1(\R^d)$ such that $u_n \wto u_0$ weakly in $\dot H^1(\R^d)$, $u_n\to u_0$ strongly in $L^2_{\rm loc}(\R^d)$ and $u_n(x)\to u_0(x)$ for a.e. $x\in \R^d$.  
We will first prove that $u_0\ne 0$, which is the key difficulty. Then the  optimality of $u_0$ is standard.  

\bigskip

\noindent
{\bf Step 1:} We prove that $u_0\ne 0$. Assume by contradiction that $u_0=0$. Fix a smooth partition of unity $\chi^2+\eta^2=1$ with $\1_{B_1}\le \chi \le \1_{B_2}$. Then we have 
\begin{align}\label{eq:IMS-p2}
\int_{\R^d} \left(  |\nabla u_n|^2 - |\nabla (\chi u_n)|^2 - |\nabla (\eta u_n)|^2  \right) = - \int_{\R^d} (|\nabla \chi|^2 + |\nabla \eta|^2 )|u_n|^2 \to 0
\end{align} 
since $\1_{B_2} u_n \to 0$ strongly in $L^2(\R^d)$. Consequently, by the Sobolev inequality, 
\begin{align}
\label{eq:S-p2-0}
1= S_d^{-1}\lim_{n\to \infty} \int_{\R^d} \left(  |\nabla (\chi u_n)|^2 + |\nabla (\eta u_n)|^2  \right) \ge \limsup_{n\to \infty} \left( \|\chi u_n\|^2_{L^{2^*}} + \|\eta u_n\|^2_{L^{2^*}} \right). 
\end{align}
On the other hand, note that by \eqref{eq:proof-simple=p2-input-1} and \eqref{eq:proof-simple=p2-input-2}, 
$$\|\chi u_n\|^{2^*}_{L^{2^*}} \ge \int_{B_1}|u_n|^{2^*} = \delta_d, \quad \|\eta u_n\|^{2^*}_{L^{2^*}} \ge \int_{B_2^c} |u_n|^2 \ge \frac{1}{2}.$$ 
Therefore, using \eqref{eq:strict-binding} (with $p=2$) and combining with \eqref{eq:S-p2-0} we find that 
$$
1 \ge  (1-\eps)^{-2/2^*} \limsup_{n\to \infty} \left( \|\chi u_n\|^{2^*}_{L^{2^*}} + \|\eta u_n\|^{2^*}_{L^{2^*}} \right)^{2/2^*} .
$$
Since $\int_{\R^d} |u_n|^{2^*}=1$,  we deduce that 
\begin{align}\label{eq:1-chi-eta}
\liminf_{n\to \infty} \int_{\R^d} (1-\chi^{2^*}-\eta^{2^*}) |u_n|^{2^*} \ge  \eps. 
\end{align}

Next, for every $L\in \mathbb{N}$, we can find a family of smooth functions $\chi_k^2+\eta_k^2=1$ for $k=1,2,...,L$ such that $\{1-\chi_k^{2^*}-\eta_k^{2^*}\}_{k=1}^L$ have disjoint supports in $B_2\backslash B_1$ (e.g. we  choose $\chi_k=1$ in $\{|x|<1+(k-1)/L\}$ and $\chi_k=0$ in $\{|x|>1+k/L\}$). Since 
$$\sum_{k=1}^L (1-\chi_k^{2^*}-\eta_k^{2^*}) \le \1_{B_2\backslash B_1},$$
by applying \eqref{eq:1-chi-eta} with $(\chi,\eta)$ replaced by $(\chi_k,\eta_k)$, we find that 
\begin{align}\label{eq:p2-nontrivial-weak-limit-final}
1 \ge  \liminf_{n\to \infty}\int_{B_2\backslash B_1}  |u_n|^{2^*} \ge  \liminf_{n\to \infty} \sum_{k=1}^L \int_{\R^d} (1-\chi_k^{2^*}-\eta_k^{2^*})  |u_n|^{2^*} \ge L\eps. 
\end{align}
By taking $L$ large enough such that $L\eps>1$, we get a contradiction. Thus  $u_0\ne 0$.

\bigskip
\noindent
{\bf Step 2:} Now we prove that $u_0$ is an optimizer and $u_n\to u_0$ strongly in $\dot H^1(\R^d)$ by using \eqref{eq:S-p2-1}. 
Since $u_n\to u_0$ pointwise almost everywhere, by the Brezis--Lieb lemma \cite{BreLie-83},  
\begin{align}\label{eq:S-p2-1b}
\int_{\R^d} \left(  |u_n|^{2^*} - |u_0|^{2^*} - |u_n-u_0|^{2^*}  \right) \to 0. 
\end{align}
Using the Sobolev inequality (twice, for $u_0$ and $u_n-u_0$)
and the (non-strict)  inequality 
\begin{align}\label{eq:non-strict-binding}
a^{2/2^*}+b^{2/2^*}  \ge (a+b)^{2/2^*},\quad \forall a,b\in [0,1],
\end{align}
we deduce from \eqref{eq:S-p2-1} and \eqref{eq:S-p2-1b} that 
\begin{align}\label{eq:S-p2-final}
1 &= S_d^{-1}  \lim_{n\to \infty} \int_{\R^d}  |\nabla u_n|^2 = S_d^{-1}  \lim_{n\to \infty} \left( \int_{\R^d}  |\nabla u_0|^2 +     \int_{\R^d}  |\nabla (u_n- u_0)|^2  \right) \\
& \ge   \liminf_{n\to \infty} ( \| u_0\|_{L^{2^*}}^2 +  \|u_n- u_0\|_{L^{2^*}}^2)  \ge \liminf_{n\to \infty} \left( \|  u_0\|_{L^{2^*}}^{2^*}+ \|u_n- u_0\|_{L^{2^*}}^{2^*}  \right)^{2/2^*} =1. \nonumber
\end{align}
In particular, we conclude that $\|\nabla u_0\|^2_{L^2}= S_d \| u_0\|_{L^{2^*}}^2$, namely $u_0$ is an optimizer for \eqref{eq:Sobolev}. Moreover, if $\|u_n-u_0\|_{L^{p*}} \ge \delta>0$ then we can use the strict inequality \eqref{eq:strict-binding} instead of \eqref{eq:non-strict-binding}, leading to a contradiction. Thus $\|u_n-u_0\|_{L^{2^*}}^2\to 0$. This implies $\|\nabla (u_n- u_0)\|_{L^2} \to 0$ due to \eqref{eq:S-p2-final}. 
\end{proof}


%
%

\section{Proof of Theorem \ref{thm:1}} \label{sec:general-p} 

Let us start with the proof of Lemma \ref{lem:concentration} which is of independent interest. 

\begin{proof}[Proof of Lemma \ref{lem:concentration}]  Assume that \eqref{eq:p-vn-delta} does not hold, namely
\begin{align}\label{eq:vn-eps-rn}
\eps_{r,n} := \sup_{y\in \R^d}\int_{B_r(y)} |v_n|^{p^*} \quad \text{ satisfies } \quad \lim_{r\to 0} \limsup_{n\to \infty} \eps_{r,n}=0.
\end{align}
Take a fixed function $\varphi\in C_c^\infty(\R^d)$ such that $\1_{B_1}\le \varphi\le \1_{B_2}$ and denote 
$$
\varphi_{r,y}(x) = \varphi((x-y)/r).
$$
Since $B_{2r}(y)$ can be covered by finitely many ball of radius $r$, we have
\begin{align}\label{eq:vn-eps-2rn}
\sup_{y\in \R^d}\int_{\R^d} |\varphi_{r,y} v_n|^{p^*}  \le \sup_{y\in \R^d}\int_{B_{2r}(y)}|v_n|^{p^*} \lesssim_d \sup_{y\in \R^d}\int_{B_{r}(y)}|v_n|^{p^*}   \lesssim_d  \eps_{r,n}. 
\end{align}
Combining \eqref{eq:vn-eps-2rn} and the Sobolev inequality \eqref{eq:Sobolev} we find that 
\begin{align}\label{eq:vn-Sobolev-phi}
\int_{\R^d} |\varphi_{r,y}  v_n|^{p^*} &\lesssim_d  \eps_{r,n}^{1-p/p^*}  \left( \int_{\R^d} |\varphi_{r,y}  v_n|^{p^*} \right)^{p/p^*} \lesssim_d  \eps_{r,n}^{1-p/p^*} \int_{\R^d} |\nabla (\varphi_{r,y} v_n)|^p \nn\\
&\lesssim_d  \eps_{r,n}^{1-p/p^*}  \int_{\R^d} (|\varphi_{r,y}|^p |\nabla v_n|^p + |\nabla \varphi_{r,y}|^p |v_n|^p)
\end{align}
for all $y\in \R^d$. Next, we integrate over $y\in B_{2R}$ with $R>r$ and use  
\begin{align*}
r^{d}\1_{B_R}(x) \lesssim_d  \int_{B_{2R}} |\varphi_{r,y}(x)|^{p^*} \d y  & \le \int_{B_{2R}} |\varphi_{r,y}(x)|^{p}  \d y \lesssim_d r^d \1_{B_{4R}}(x), \\
 \int_{\R^d} |\nabla \varphi_{r,y}(x)|^p  \d y &\lesssim_{d} r^{d-p} \1_{B_{4R}}(x)
\end{align*}
for all $x\in \R^d$. By Fubini's theorem, we conclude from \eqref{eq:vn-Sobolev-phi} that  
\begin{align}\label{eq:vn-Sobolev-phi-conclusion-1111}
\int_{B_R} |v_n|^{p^*} \lesssim_d  \eps_{r,n}^{1-p/p^*}  \left( \int_{B_{4R}}  |\nabla v_n|^p + r^{-p} \int_{B_{4R}} |v_n|^p\right). 
\end{align}
Since $|\nabla v_n|$ is bounded in $L^p(\R^d)$ and $v_n\to 0$ strongly in $L^p_{\rm loc}(\R^d)$ by the Sobolev embedding, from \eqref{eq:vn-Sobolev-phi-conclusion-1111} we can take $n\to 0$ and then  $r\to 0$ to obtain \eqref{eq:vn-Sobolev-phi-conclusion-1}. \end{proof}

\begin{proof}[Proof of Theorem \ref{thm:1} for $1< p<d$] 
For any minimizing sequence $\{u_n\}_{n=1}^\infty \subset \dot{W}^{1,p}(\R^d)$ for \eqref{eq:Sobolev}, up to a translation and a dilation, we can assume that
\begin{align} \label{eq:p-un-1}
\int_{\R^d} |\nabla u_n|^p \to S=S_{d,p}, \quad \int_{\R^d} |u_n|^{p^*}=1, \quad \int_{B_1} |u_n|^{p^*} =  \sup_{x\in \R^d} \int_{B_1(x)} |u_n|^{p^*} = \frac 1 2. 
\end{align}
By the Banach--Alaoglu theorem and the Sobolev embedding, up to a subsequence, we can assume that $u_n \wto u_0$ weakly in $\dot W^{1,p}(\R^d)$, strongly in $L^p_{\rm loc}(\R^d)$ and pointwise almost everywhere in $\R^d$.  We will prove that $u_0$ is an optimizer.

\bigskip
\noindent
{\bf Step 1:} We prove that the sequence $|u_n|^{p^*}$ is tight, namely 
\begin{align}\label{eq:p-un-tight} 
\lim_{R\to \infty} \limsup_{n\to \infty} \int_{B_R^c} |u_n|^{p^*} =0,
\end{align} 
and that $|u_n|^{p^*}$ does not have any delta-type concentration, namely 
\begin{align}\label{eq:p-un-no-delta} 
\lim_{r\to 0} \limsup_{n\to \infty} \sup_{x\in \R^d}\int_{B_r(x)} |u_n|^{p^*} =0. 
\end{align} 

Assume by contradiction that either \eqref{eq:p-un-tight} or \eqref{eq:p-un-no-delta} fails to hold. Then, up to a subsequence of $\{u_n\}$, there are $\delta>0$, $\{x_n\}\subset \R^d$ and $\{R_n\}\subset (0,\infty)$ such that  
\begin{align}\label{eq:p-un-gen-dichotomy} 
1-\delta \ge \int_{|x-x_n|<R_n} |u_n(x)|^{p^*} \d x \ge \delta,\quad \lim_{n\to \infty} \int_{R_n< |x-x_n|<2 R_n}  |u_n(x)|^{p^*} \d x =0. 
\end{align} 
In fact, if $\{u_n\}$ is not tight, namely  \eqref{eq:p-un-tight} fails, then up to a subsequence of $\{u_n\}$, 
\begin{align}\label{eq:p-un-tight-not} 
\int_{|x|> {2^{2n}}} |u_n|^{p^*} \ge \delta>0,\quad \forall n\in \mathbb{N}. 
\end{align} 
Thus \eqref{eq:p-un-gen-dichotomy} holds with $x_n=0$ and a suitable choice $R_n\in \{2^n, 2^{n+1},..., 2^{2n-1}\}$ such that 
\begin{align} \label{eq:p-un-tight-not-choice-Rn}
\int_{R_n<|x|<2 R_n} |u_n|^{p^*} \le \frac{1}{n} \sum_{k=n}^{2n-1}  \int_{2^k< |x|<2^{k+1}}  |u_n|^{p^*} = \frac{1}{n} \int_{2^n < |x|<2^{2n}}  |u_n|^{p^*}\le \frac{1}{n} \to 0.
\end{align} 
Similarly, if there is a delta-type concentration, namely \eqref{eq:p-un-no-delta} fails, then up to a subsequence of $\{u_n\}$, there exists a sequence $\{x_n\} \subset \R^d$ such that
\begin{align}\label{eq:p-un-no-delta-second} 
\int_{|x_n-x|<2^{-2n} } |u_n(x) |^{p^*} \d x \ge \delta >0,\quad \forall n\in \mathbb{N}. 
\end{align} 
We can verify \eqref{eq:p-un-gen-dichotomy} by choosing $R_n\in \{2^{-2n}, 2^{-2n+1},..., 2^{-n-1}\}$ similarly to \eqref{eq:p-un-tight-not-choice-Rn}. 
%

\medskip

Now let us derive a contradiction from \eqref{eq:p-un-gen-dichotomy}. Choose a smooth partition of unity $\chi^p+\eta^p=1$ with $\1_{B_1}\le \chi \le \1_{B_2}$ and denote 
\begin{align}\label{eq:def-chi-n-eta-n}
\chi_n(x)= \chi((x-x_n)/R_n), \quad \eta_n(x)= \eta((x-x_n)/R_n). 
\end{align}
Using 
H\"older's inequality, $\|\nabla \chi_n\|_{L^d(\R^d)}= \| \nabla \chi\|_{L^d(\R^d)}$ and \eqref{eq:p-un-gen-dichotomy}, we have 
\begin{align*}
\int_{\R^d} |\nabla \chi_n|^p |u_n|^p \le \|\nabla \chi_n\|_{L^d(\R^d)}^p \left( \int_{R_n<|x-x_n|<2R_n} |u_n|^{p^*}\right)^{p/p^*} \to 0.
\end{align*}
Therefore, by the triangle inequality 
\begin{align*}
\int_{\R^d} |\nabla (\chi_n u_n) |^p  =  \int_{\R^d} |\chi_n|^p | \nabla u_n|^p +o(1)_{n\to \infty}. 
\end{align*}
Combining with a similar formula for $\eta_n$ and using $\chi_n^p+\eta_n^p=1$, we obtain 
\begin{align}
\label{eq:S-p-1}
\int_{\R^d} |\nabla u_n|^p =  \int_{\R^d}  |\nabla (\chi_n u_n) |^p + \int_{\R^d}  |\nabla (\eta_n u_n)|^p  + o(1)_{n\to \infty}.  
\end{align}

Note that by \eqref{eq:p-un-gen-dichotomy} we also have 
\begin{align}
\label{eq:S-p-0}
\int_{\R^d} |\chi_n u_n|^{p^*} \ge \delta>0, \quad \int_{\R^d} |\eta_n u_n|^{p^*} \ge \delta + o(1)_{n\to \infty}.
\end{align}
Using \eqref{eq:S-p-1}, the Sobolev inequality (twice), and the strict inequality \eqref{eq:strict-binding} thanks to \eqref{eq:S-p-0}, we get
\begin{align}\label{eq:S-p-2}
1  &= S^{-1}\lim_{n\to \infty} \int_{\R^d} |\nabla u_n|^p = S^{-1}\lim_{n\to \infty} \left( \int_{\R^d}   |\nabla (\chi_n u_n) |^p  + \int_{\R^d} |\nabla (\eta_n u_n)|^p  \right) \nn \\
&\ge  \liminf_{n\to \infty}  \left( \|\chi_n u_n\|_{L^{p^*}}^{p} +  \|\eta_n u_n\|_{L^{p^*}}^{p}  \right) \nn\\
& \ge  \liminf_{n\to \infty}  (1-\eps)^{-p/p^*} \left( \|\chi_n u_n\|_{L^{p^*}}^{p^*}  + \|\eta_n u_n\|_{L^{p^*}}^{p^*} \right)^{p/p^*} = (1-\eps)^{-p/p^*}
\end{align}
for a constant $\eps\in (0,1)$ depending only on $\delta$ and $p/p^*$. Here in the last equality we used the second part of \eqref{eq:p-un-gen-dichotomy}. 
This is a contradiction since $(1-\eps)^{-p/p^*} >1$. Thus \eqref{eq:p-un-gen-dichotomy} cannot occur, and hence both  \eqref{eq:p-un-tight} and \eqref{eq:p-un-no-delta} hold true. 

\bigskip
\noindent
{\bf Step 2:} We consider the sequence $v_n=u_n - u_0$, which converges weakly to $0$ in $\dot{W}^{1,p}(\R^d)$ and converges to $0$ pointwise almost everywhere. By the Brezis--Lieb lemma, for any open set $\Omega\subset \R^d$ we have
\begin{align}\label{eq:Brezis--Lieb-vn}
 \int_{\Omega} |u_n|^{p^*} =  \int_{\Omega} |u|^{p^*} +  \int_{\Omega} |v_n|^{p^*} + o(1)_{n\to \infty}. 
\end{align}
By Step 1, the sequence $|u_n|^{p^*}$ is tight and has no delta-type concentration. Hence, \eqref{eq:Brezis--Lieb-vn} implies that $|v_n|^{p^*}$ is also tight and has no delta-type concentration. By Lemma \ref{lem:concentration}, $v_n\to 0$ strongly in $L^{p^*}(\R^d)$. Thus $u_n \to u_0$ strongly in $L^{p^*}(\R^d)$, and we find that 
$$
1=\lim_{n\to \infty}\|u_n\|_{L^{p^*}(\R^d)}^p =  \|u_0\|_{L^{p^*}(\R^d)}^p \le  S^{-1}\|\nabla u_0\|_{L^p(\R^d)}^p\le \liminf_{n\to \infty}  S^{-1}\|\nabla u_n\|_{L^p(\R^d)}^p =1.
$$

Here the last inequality follows from the assumption that $\nabla u_n \wto \nabla u_0$ weakly in $L^p(\R^d)$. This implies that  $u_0$ is an optimizer for \eqref{eq:Sobolev}. Moreover, $\|\nabla u_n\|_{L^p} \to \|\nabla u_0\|_{L^p}$ by the above argument, and hence $\nabla u_n \to \nabla u_0$ strongly for $L^p(\R^d)$. 
\end{proof}

\begin{remark} We can interpret \eqref{eq:p-un-gen-dichotomy} as a ``generalized dichotomy". More precisely, in Lions' terminology \cite{Lions-84}, the ``vanishing case" is already ruled out by the assumption \eqref{eq:p-un-1}. Hence, if $|u_n|^{p^*}$ is not tight, we have the ``standard dichotomy". Our observation is that the case of delta-type concentration can be treated similarly to the "standard dichotomy", namely both cases can be captured by \eqref{eq:p-un-gen-dichotomy} with the only difference being that $R_n\to \infty$ and $R_n \to 0$, respectively. This approach allows us to avoid having to discuss the limit of $|\nabla u_n|^p$ and $|u_n|^{p^*}$ as measures, as in \cite{Lions-85}. 
\end{remark}

\section{Proof of Theorem \ref{thm:2}} \label{sec:general-p-s} 

In this section we prove Theorem \ref{thm:2}. Recall that $d>s>0$, $d/s>p>1$ and $p^*=dp/(d-ps)$. We consider the cases $s\in \mathbb{N}$ and $s\notin \mathbb{N}$ separately.

\subsection{The integer case $s\in \mathbb{N}$} We have the following extension of Lemma \ref{lem:concentration}. 

\begin{lemma}\label{lem:concentration-s} Let $d\in \mathbb{N}$, $s>0$ and $1<p<d/s$. Assume that $\{v_n\}$ is bounded in $\dot W^{s,p}(\R^d)$ such that $v_n \wto 0$ weakly in $\dot{W}^{s,p}(\R^d)$. Then \eqref{eq:vn-Sobolev-phi-conclusion-1} or \eqref{eq:p-vn-delta} holds. 
\end{lemma}

\begin{proof}
The proof goes similarly. We prove that if \eqref{eq:p-vn-delta} does not hold, then \eqref{eq:vn-Sobolev-phi-conclusion-1} holds. The only nontrivial change is in \eqref{eq:vn-Sobolev-phi} where we use now
\begin{align}\label{eq:vn-Sobolev-phi-s}
\int_{\R^d} |\varphi_{r,y}  v_n|^{p^*} &\lesssim_d  \eps_{r,n}^{1-p/p^*}  \left( \int_{\R^d} |\varphi_{r,y}  v_n|^{p^*} \right)^{p/p^*} \lesssim_d  \eps_{r,n}^{1-p/p^*} \int_{\R^d} |D^s (\varphi_{r,y} v_n)|^p \nn\\
&\lesssim_d  \eps_{r,n}^{1-p/p^*}  \sum_{j=0}^s \int_{\R^d} |D^j \varphi_{r,y}(x)|^p | D^{s-j}v_n(x)|^p \d x. 
\end{align}
Then we integrate over $y\in B_{2R}$ with $R>r$ and obtain 
\begin{align}\label{eq:vn-Sobolev-phi-conclusion}
\int_{B_R} |v_n|^{p^*} \lesssim_d  \eps_{r,n}^{1-p/p^*}  \int_{B_{4R}}  |D^s v_n|^p + \sum_{j=1}^{s} r^{-jp} \int_{B_{4R}} |D^{s-j}v_n|^p. 
\end{align}
By the Sobolev embedding, $\int_{B_{4R}} |D^{s-j}v_n|^p \to 0$ when $n\to \infty$. Thus we can take $n\to \infty$, then $r\to 0$ and obtain \eqref{eq:vn-Sobolev-phi-conclusion-1} as desired. 
\end{proof}

\begin{proof}[Proof of Theorem \ref{thm:2} for $s\in \mathbb{N}$] We consider a minimizing sequence $\{u_n\}_{n=1}^\infty \subset \dot{W}^{s,p}(\R^d)$ for \eqref{eq:Sobolev-fractional} such that 
\begin{align} \label{eq:p-un-1-s}
\|D^s u_n\|_{\dot W^{s,p}(\R^d)}^p \to S_{d,s,p}, \quad \int_{\R^d} |u_n|^{p^*}=1, \quad \int_{B_1} |u_n|^{p^*} =  \sup_{x\in \R^d} \int_{B_1(x)} |u_n|^{p^*} = \frac 1 2.
\end{align}
Moreover, up to subsequences and symmetries, we can assume that $u_n \wto u_0$ weakly in $\dot W^{s,p}(\R^d)$, strongly in $L^p_{\rm loc}(\R^d)$ and pointwise almost everywhere in $\R^d$.  

First, we can show that $|u_n|^{p^*}$ is tight and has no delta-type concentration. By arguing as in the proof of Theorem \ref{thm:1}, if either \eqref{eq:p-un-tight} or \eqref{eq:p-un-no-delta} fails, then \eqref{eq:p-un-gen-dichotomy} must hold. Then we need the following analogue of \eqref{eq:S-p-1}:
 \begin{align} \label{eq:S-p-1-s}
\int_{\R^d} \left(  |D^s u_n|^p - |D^s (\chi_n u_n) |^p - |D^s (\eta_n u_n)|^p  \right) \to 0
\end{align}
with $\chi_n,\eta_n$ in \eqref{eq:def-chi-n-eta-n}. In fact, by the triangle inequality and H\"older's inequality 
  we get 
\begin{align*}
&\left| \| D^s (\chi_n u_n)  \|_{L^p(\R^d)} - \| \chi_n D^s u_n\|_{L^p(\R^d)} \right|^p \lesssim_d \sum_{j=1}^s \int_{\R^d} |D^j \chi_n|^p |D^{s-j} (\tilde \chi_n u_n)  |^p \\
&\le  \sum_{j=1}^s  \left(  \int_{\R^d} |D^j \chi_n|^{d/j} \right)^{jp/d}  \left( \int_{\R^d} |D^{s-j} (\tilde \chi_n u_n)  |^{dp/(d-jp)} \right)^{1-jp/d}\\
&\lesssim_{d,s,p}     \sum_{j=1}^s  \left(  \int_{\R^d} |D^j \chi|^{d/j} \right)^{jp/d}  \| \tilde \chi_n  u_n \|_{L^{p^*}(\R^d)}^{p\theta_j} \| D^{s} (\tilde \chi_n  u_n) \|_{L^{p}(\R^d)}^{p(1-\theta_j)}  \to 0
\end{align*}
 Here we took $\tilde  \chi_n(x)=\tilde \chi(x/R_n)$ with $R_n$ in \eqref{eq:p-un-gen-dichotomy} and a fixed smooth function $\tilde \chi$ such that $\1_{B_{2}\backslash B_1} \le  \tilde \chi \le \1_{B_{3}}$ (so that $\tilde \chi_n=1$ in the support of $|\nabla \chi_n|$). We also used $\|D^j \chi_n\|_{L^{d/j}}=\|D^j \chi\|_{L^{d/j}}$ and the Gagliardo--Nirenberg inequality \cite{Gagliardo-59,Nirenberg-59} 
 $$
 \| D^{s-j} f \|_{L^{dp/(d-jp)}(\R^d)} \lesssim_{d,s,p} \|f\|^{\theta_j}_{L^{p^*} (\R^d)}  \| D^s f \|^{1-\theta_j}_{L^{p}(\R^d)}
 $$
 with suitable $\theta_j\in (0,1)$ for all $j\in \{0,1,2,...,s\}$. Combining with a similar bound for $\eta_n$, we obtain \eqref{eq:S-p-1-s}. From  \eqref{eq:p-un-gen-dichotomy} and \eqref{eq:S-p-1-s}, we deduce a contradiction similar to \eqref{eq:S-p-2}.

Thus, both  \eqref{eq:p-un-tight} and \eqref{eq:p-un-no-delta} hold true. Consequently, $v_n=u_n-u_0 \to 0$ strongly in $L^{p^*}(\R^d)$ by Lemma \ref{lem:concentration-s}. Therefore, we conclude that $u_0$ is a minimizer and $u_n\to u_0$ strongly in $\dot W^{s,p}(\R^d)$, as in the proof of Theorem \ref{thm:1}. 
\end{proof}

\subsection{The fractional case $s\notin \mathbb{N}$} Assume that $s=m+\sigma$ with $m\in \{0,1,2,...\}$ and $\sigma\in (0,1)$. This case is indeed easier since we can follow the strategy in Section \ref{sec:p=2}. The only additional complication is the following analogue of \eqref{eq:IMS-p2}, which will be proved later. 

\begin{lemma}\label{lem:IMS} If $u_n\wto 0$ in $\dot W^{s,p}(\R^d)$, then for every partition function $\chi^p+\eta^p=1$ with $\1_{B_1}\le \chi \le \1_{B_2}$ we have 
$$
\| u_n \|_{\dot W^{s,p}(\R^d)}^p = \| \chi u_n \|_{\dot W^{s,p}(\R^d)}^p +  \| \eta u_n \|_{\dot W^{s,p}(\R^d)}^p + o(1)_{n\to \infty}. 
$$
\end{lemma}

\begin{proof}[Proof of Theorem \ref{thm:2} for $s\notin \mathbb{N}$] Let $u_n$ be a minimizing sequence such that 
\begin{align} \label{eq:p-s-input}
\int_{\R^d} |u_n|^{p^*}=1, \quad \int_{B_1} |u_n|^{p^*} =  \sup_{x\in \R^d} \int_{B_1(x)} |u_n|^{p^*} = \delta_d>0, \quad \int_{B_2^c} |u_n|^{p^*} \ge \frac{1}{2}
\end{align}
for all $n\in \mathbb{N}$ (c.f. \eqref{eq:proof-simple=p2-input-1} and \eqref{eq:proof-simple=p2-input-2}). We also assume that $u_n \wto u_0$ weakly in $\dot W^{s,p}(\R^d)$ and $D^m u_n(x)\to D^m u_0(x)$ pointwise almost everywhere in $\R^d$.  

\medskip

If $u_0=0$, then we can apply Lemma \ref{lem:IMS} and the Sobolev inequality \eqref{eq:Sobolev-fractional} to bound 
\begin{align*}
1= S_{d,s,p}^{-1}\lim_{n\to \infty} \int_{\R^d} \left(  \| \chi u_n \|_{\dot W^{s,p}(\R^d)}^p + \| \eta u_n\|_{\dot W^{s,p}(\R^d)}^p  \right) \ge \limsup_{n\to \infty} \left( \|\chi u_n\|^p_{L^{p^*}} + \|\eta u_n\|^p_{L^{p^*}} \right). 
\end{align*}
Using \eqref{eq:p-s-input} and the strict inequality \eqref{eq:strict-binding}, we can find a suitable choice $(\chi, \eta)$ such that $1-\chi-\eta$ is supported in $B_2\backslash B_1$ to get a contradiction as in \eqref{eq:p2-nontrivial-weak-limit-final}.  Thus  $u_0\ne 0$. 

Since $D^m u_n(x)\to D^m u_0(x)$ pointwise,  by applying the Brezis--Lieb lemma for $(D^m u_n(x) - D^m u_n(y))/ |x-y|^{d/p+s}$ in $L^p(\R^d\times \R^d)$, we have
\begin{align*}
&\| u_n\|_{\dot W^{s,p}(\R^d)}^p - \| u_0 \|_{\dot W^{s,p}(\R^d)}^p - \| u_n - u_0 \|_{\dot W^{s,p}(\R^d)}^p \nonumber\\
&=\| D^m u_n\|_{\dot W^{\sigma,p}(\R^d)}^p - \| D^m u_0 \|_{\dot W^{\sigma ,p}(\R^d)}^p - \| D^m u_n - D^m u_0 \|_{\dot W^{\sigma,p}(\R^d)}^p \to 0. 
\end{align*}
Arguing similarly \eqref{eq:S-p2-final}, we conclude that $u_0$ is an optimizer and $u_n\to u_0$ strongly in $\dot W^{s,p}(\R^d)$. 
\end{proof}

\begin{proof}[Proof of Lemma \ref{lem:IMS}] It suffices to prove that
\begin{align}\label{eq:IMS-p-1}
\| D^m u_n \|_{\dot W^{\sigma,p}(\R^d)}^p  = \| \chi D^m u_n \|_{\dot W^{\sigma,p}(\R^d)}^p +  \| \eta D^m u_n \|_{\dot W^{\sigma,p}(\R^d)}^p  + o(1)_{n\to \infty} 
 \end{align}
and for $f\in \{\chi,\eta\}$, 
\begin{align}\label{eq:IMS-p-2}
\| D^m (fu_n) \|_{\dot W^{\sigma,p}(\R^d)}^p =   \| f D^m u_n \|_{\dot W^{\sigma,p}(\R^d)}^p + o(1)_{n\to \infty} . 
 \end{align}

First, consider \eqref{eq:IMS-p-1}. Since $f\in \{\chi,\eta\}$ is bounded and Lipschitz, we have  
\begin{align} \label{eq:int-chi-nonlocal}
\int_{\R^d} \frac{| f(x)- f(y) |^p} {|x-y|^{d+p\sigma}} \d y \lesssim_f \int_{\R^d} \frac{\min (1,|x-y|)^p} {|x-y|^{d+p\sigma}} \d y =  \int_{\R^d} \frac{\min (1,|z|)^p} {|z|^{d+p\sigma}} \d z  \lesssim 1.
\end{align}
Consequently, 
\begin{align*}
\int_{|x|< |y|} \frac{| (f(x)- f(y))D^m u_n(x)|^p}{|x-y|^{d+p\sigma}} \d x \d y &\le \int_{|x|<2} \int_{\R^d}  \frac{ |f(x)- f(y)|^p |D^m u_n(x)|^p}{|x-y|^{d+p\sigma}} \d x \d y\\
&\lesssim \int_{|x|<2} |D^m u_n(x)|^p \d x \to 0. 
\end{align*}
By the triangle inequality, the latter implies that  
\begin{align*}
& \int_{|x|<|y|} \frac{| f(y) D^m u_n(x) - f(y)  D^m u_n(y)|^p}{|x-y|^{d+p\sigma}} \d x \d y  \\
&= \int_{\R^d} \int_{|x|<|y|} \frac{| f(x) D^m u_n(x) - f(y)  D^m u_n(y)|^p}{|x-y|^{d+p\sigma}} \d x \d y + o(1)_{n\to \infty}. 
\end{align*}
Summing these formulas over $f\in \{\chi,\eta\}$ and using $\chi^p(x)+\eta^p(x)=1$ we get
\begin{align*}
&\int_{|x|<|y|} \frac{| D^m u_n(x) -D^m u_n(y)|^p}{|x-y|^{d+p\sigma}} \d x \d y \\
&= \sum_{f\in \chi,\eta} \int_{|x|<|y|} \frac{| f(x) D^m u_n(x) - f(y)  D^m u_n(y)|^p  }{|x-y|^{d+p\sigma}} \d x \d y  + o(1)_{n\to \infty}. 
\end{align*}
We have the same formula for the domain $\{|y|<|x|\}$. Therefore, \eqref{eq:IMS-p-1} holds. 

Next, consider  \eqref{eq:IMS-p-2}. The components of the vector $D^m (f u_n) - f D^m u_n$ contain finite functions of the form $f_\alpha D^\alpha u_n$ with $1\le |\alpha|\le m-1$, where $f_\alpha$ is involving the derivatives of order $m-|\alpha|$ of $f\in \{\chi,\eta\}$. In particular, $f_\alpha$ is bounded, Lipschitz and supported in $\overline{B_2}\backslash B_1$.  Therefore, we can write $f_\alpha D^\alpha u_n = f_\alpha D^\alpha ( g u_n)$ for a fixed function $g\in C_c^\infty(\R^d)$ such that $g=1$ in $B_2$. By the triangle inequality, we have 
$$
 \| D^m (f u_n) - f D^m u_n \|_{\dot W^{\sigma,p}(\R^d)} \lesssim  \sum_{0 \le |\alpha|\le m-1} \|  f_\alpha D^\alpha (g u_n) \|_{\dot W^{\sigma,p}(\R^d)}.  
$$
By following the proof of \eqref{eq:IMS-p-1} and using \eqref{eq:int-chi-nonlocal} with $f$ replaced by $f_\alpha$, we get
 \begin{align*}
& \|  f_\alpha D^\alpha (g u_n) \|_{\dot W^{\sigma,p}(\R^d)}^p =  \int_{\R^d}\int_{\R^d} \frac{ | f_\alpha(x) D^\alpha (g u_n)(x) - f_\alpha(x) D^\alpha (g u_n)(y) |^p}{|x-y|^{d+p\sigma}} \d x \d y \\
 &\lesssim_p  \int_{\R^d}\int_{\R^d} \frac{ |f_\alpha(x)-f_\alpha(y)|^p |D^\alpha (g u_n)(x)|  + |f^\alpha(y)|^p |D^\alpha (g u_n)(x) - D^\alpha (g u_n)(y))|^p}{|x-y|^{d+p\sigma}} \d x \d y\\ 
 &\lesssim  \int_{\R^d }  |D^\alpha (g u_n)(x)|^p  \d x + \|f^\alpha\|^p_{L^\infty} \| D^\alpha (g u_n)\|_{\dot W^{\sigma,p}(\R^d)}^p \to 0. 
 \end{align*}
Here we used the Sobolev embedding to get $g u_n\to 0$ strongly in  $W^{m',p}(\R^d)$ for all $m'=0,1,...,m-1$. Thus \eqref{eq:IMS-p-2} follows. The proof of Lemma \ref{lem:IMS} is complete. 
\end{proof}


\begin{thebibliography}{10}

\bibitem{Aubin-76} T. Aubin. Espaces de Sobolev sur les variétés riemanniennes. {\em Bull. Sciences Mathématiques} 100 (1976), pp. 149--173. 

\bibitem{BahCohKoc-11} H. Bahouri, A. Cohen and G. Koch. A general wavelet-based profile decomposition in the critical embedding of function spaces. 
 {\em Confluentes Mathematici} 3 (2011), pp. 387--411.

\bibitem{BreLie-83}  H. Brézis and E. Lieb. A relation between pointwise convergence of functions and convergence of functionals. {\em Proc. Amer. Math. Soc.}  88 (1983), no. 3, pp. 486--490.



\bibitem{CarCarLos-10}  E. A. Carlen, J. A. Carrillo and M. Loss. Hardy-Littlewood-Sobolev inequalities via fast diffusion
flows. {\em Proc. Nat. Acad. USA}  107 (2010), no. 46, pp. 19696--19701. %


\bibitem{Frank-11} R.L. Frank. Sobolev inequalities and uncertainty principles in quantum mechanics. Part 1. {\em Lecture Notes 2011, LMU Munich}. Online available at \url{https://www.math.lmu.de/~frank/sobweb1.pdf}. 

\bibitem{FraLos-22} R.L. Frank and M. Loss. Existence of optimizers in a Sobolev inequality for vector fields. {\em Ars Inveniendi Analytica} (2022), Paper No. 1, 31 pp. 
 
\bibitem{FraLie-10}  R. L. Frank and E. H. Lieb. Inversion positivity and the sharp Hardy-Littlewood-Sobolev inequality. {\em Calc. Var. PDE} 39 (2010), no. 1--2, pp. 85--99.

\bibitem{FraLie-12} R.L. Frank and E.H. Lieb. Sharp constants in several inequalities on the Heisenberg group. {\em Ann. Math.} 176 (2012), 349--381.

\bibitem{FraLie-12b} R.L. Frank and E.H. Lieb. A new rearrangement-free proof of the sharp Hardy--Littlewood--Sobolev inequality. In: Brown, B.M. (ed.) {\em Operator Theory: Advances and Applications. Spectral Theory, Function Spaces and Inequalities}, Vol. 219, pp. 55--67. Birkhäuser, Basel (2012). 

\bibitem{Gagliardo-59} E. Gagliardo. Ulteriori proprietà di alcune classi di funzioni di più variabili. {\em Ricerche di Matematica} (1959), pp. 24--51. 

\bibitem{Gerard-98} P. G\'erard.  Description du d\'efaut de compacit\'e de l’injection de Sobolev. {\em ESAIM Control
Optim. Calc. Var.} 3 (1998), pp. 213-233. 

\bibitem{GerMeyOru-97}  P. G\'erard, Y. Meyer and F. Oru. In\'egalit\'es de Sobolev pr\'ecis\'ees. {\em S\'eminaire EDP, Exp. No.
IV}, Ecole Polytechnique, Palaiseau, 1997. 

\bibitem{Jaffard-99} S. Jaffard. Analysis of the lack of compactness in the critical Sobolev embeddings. {\em J. 
Funct. Anal.} 161 (1999), pp. 384--396. 

\bibitem{KilVis-13} R. Killip and M. Visan. Nonlinear Schrödinger Equations at Critical Regularity. {\em Clay Mathematics Proceedings.} Vol. 17, 2013.  Online available at \url{https://www.math.ucla.edu/~visan/ClayLectureNotes.pdf}. 

\bibitem{LenLew-11} E. Lenzmann and M. Lewin. On singularity formation for the $L^2$-critical Boson star equation. {\em Nonlinearity} 24 (2011), p.  3515. 

\bibitem{Lewin-10} M. Lewin. Describing lack of compactness in Sobolev spaces, from {\em Variational Methods in Quantum Mechanics}, Lecture Notes 2010, University of Cergy-Pontoise. Online available at \url{https://www.ceremade.dauphine.fr/~lewin/data/conc-comp.pdf}.

\bibitem{Lieb-83} E. H. Lieb. Sharp constants in the Hardy--Littlewood--Sobolev and related inequalities. {\em Ann. Math.} 118 (1983), pp. 349--374.

\bibitem{Lions-84} P.-L. Lions. The concentration-compactness principle in the calculus of variations. The locally compact case, Part 1. {\em Ann. Inst. Henri Poincaré Anal. Non Linéaire} 1 (1984), pp. 109--145.  


\bibitem{Lions-84b} P.-L. Lions. The concentration-compactness principle in the calculus of variations The locally compact
case. {\em Ann. Inst. Henri Poincaré Anal. Non Linéaire} 1 (1984), pp. 223--283 

\bibitem{Lions-85} P.-L. Lions. The concentration-compactness principle in the calculus of variations. The limit case, Part 1. {\em Rev.
Mat. Iberoam. 1} (1985), pp. 145--201.


\bibitem{Nirenberg-59} L. Nirenberg. On elliptic partial differential equations. {\em Annali della Scuola Normale Superiore di Pisa} 3 (13) (1959), pp. 115--162.


\bibitem{PalPis-14} G. Palatucci and A. Pisante. Improved Sobolev embeddings, profile decomposition, and concentration compactness for fractional Sobolev spaces. {\em Calc. Var. PDE} 50 (2014), pp. 799--829. 

\bibitem{Rosen-71} G.  Rosen. Minimum Value for c in the Sobolev Inequality $\|\phi^3\| \le c \|\nabla \phi\|^{3}$. {\em SIAM J. Appl. Math.} 21 (1971), pp. 30--32. 

\bibitem{Sobolev-38} S.L. Sobolev. Sur un théorème de l'analyse fonctionnelle. {\em Mat. Sbornik} 4 (1938), pp. 471-496.


\bibitem{Talenti-76} G. Talenti. Best constant in Sobolev inequality. {\em Ann. Mat. Pura Appl.} 110 (1976), pp. 353--
372.


\bibitem{Zhang-21} Y. Zhang, Optimizers of the Sobolev and Gagliardo--Nirenberg
inequalities in $\dot{W}^{s,p}$. {\em Calc. Var. PDE} (2021) 60:10.  


\end{thebibliography}
\end{document}